\documentclass{amsart}
\usepackage{amsfonts,amsthm,amsmath,amssymb,bm,mathrsfs,color,comment}

\newtheorem{theorem}{Theorem}
\newtheorem{corollary}[theorem]{Corollary}
\newtheorem{proposition}[theorem]{Proposition}
\newtheorem{remark}[theorem]{Remark}
\newtheorem{lemma}[theorem]{Lemma}

\def\re{\mathbb{R}}
\def\C{\mathbb{C}}
\def\N{\mathbb{N}}

\def\Sp{\mathbb{S}}
\def\eps{\varepsilon}
\def\pd{\partial}
\def\ol{\overline}

\def\wh{\widehat}

\def\la{\lambda}
\def\al{\alpha}
\def\th{\theta}

\def\disp{\displaystyle}
\def\({\left(}
\def\){\right)}

\def\pd{\partial}

\def\intRN{\int_{\mathbb{R}^N}}
\def\intS{\int_{\mathbb{S}^{N-1}}}

\def\bu{{\bm u}}
\def\bv{{\bm v}}
\def\bx{{\bm x}}
\def\be{{\bm e}}
\def\bsig{{\bm \sigma}}
\def\bphi{{\bm \varphi}}
\def\bvh{\widehat{{\bm v}}}

\begin{document}
\title[Hardy-Leray inequality]{Sharp Hardy-Leray and Rellich-Leray inequalities for curl-free vector fields \\}

\author[N. Hamamoto \and F. Takahashi]{Naoki Hamamoto \and Futoshi Takahashi}

\address{%
Department of Mathematics, Osaka City University \\
3-3-138 Sugimoto, Sumiyoshi-ku, Osaka 558-8585, Japan}

\email{yhjyoe@yahoo.co.jp {\rm (N.Hamamoto)}} 
\email{futoshi@sci.osaka-cu.ac.jp {\rm (F.Takahashi)}}

\begin{abstract}
In this paper, we prove Hardy-Leray and Rellich-Leray inequalities for curl-free vector fields with sharp constants.
This complements the former work by Costin-Maz'ya \cite{Costin-Mazya}
on the sharp Hardy-Leray inequality for axisymmetric divergence-free vector fields.
\end{abstract}


\date{\today}

\dedicatory{}

\maketitle

\section{Introduction}


In this paper, we concern the classical functional inequalities called Hardy-Leray and Rellich-Leray inequality for smooth vector fields
and study how the best constants will change according to the pointwise constraints on their differentials.

Let $N \in \N$ be an integer with $N \ge 2$ and put $\bx=(x_1,\cdots,x_N) \in \re^N$. 
In the following, $C_c^\infty(\re^N)^N$ denotes a set of smooth vector fields 
\[
	\bu = (u_1,u_2,\cdots, u_N): \re^N \ni \bx \mapsto \bu(\bx) \in \re^N
\]
having compact supports on $\re^N$.
Let $\gamma \ne 1 - N/2$.
Then it is well known that 
\begin{equation}
\label{Hardy}
	\(\gamma + \frac{N}{2}-1 \)^2 \intRN \frac{|\bu|^2}{|\bx|^2} |\bx|^{2\gamma} d{x} \le \intRN |\nabla \bu|^2 |\bx|^{2\gamma} d{x}
\end{equation}
holds for any vector field $\bu \in C_c^\infty(\re^N)^N$ with $\bu({\bm 0}) = {\bm 0}$ if $\gamma < 1 - N/2$.
This is a higher dimensional extension of 1-dimensional inequality by G. H. Hardy, see \cite{Hardy}, also \cite{Opic-Kufner}, 
and was first proved by J. Leray \cite{Leray} in 1933 when the weight $\gamma = 0$, 
see also the book by Ladyzhenskaya \cite{Ladyzhenskaya}.
It is also known that the constant $\(\gamma + \frac{N}{2}-1 \)^2$ is sharp and never attained.
In \cite{Costin-Mazya}, 
Costin and Maz'ya proved that if the smooth vector fields are axisymmetric and subject to the divergence-free constraint ${\rm div} \, \bu = 0$,
then the constant in \eqref{Hardy} can be improved and replaced by a larger one. 
More precisely, they proved the following:

%
%

\vspace{1em}\noindent
{\bf Theorem A.} (Costin-Maz'ya \cite{Costin-Mazya})
{\it
Let $N \ge 3$. 
Let $\gamma \neq 1-N/2$ be a real number and $\bu \in C^\infty_c(\re^N)^N$ be an axisymmetric divergence-free vector field. 
Assume that $\bu({\bm 0})={\bm 0}$ if $\gamma <1-N/2$.
Then
\[
	C_{N,\gamma} \intRN \frac{|\bu|^2}{|\bx|^2} |\bx|^{2\gamma} d{x} \le \intRN |\nabla \bu|^2 |\bx|^{2\gamma} d{x}
\]
holds with the optimal constant $C_{N,\gamma}$ given by
\[  
	C_{N,\gamma} 
	= \begin{cases}
	\( \gamma + \frac{N}{2}-1 \)^2 \dfrac{N + 1 + \(\gamma - \frac{N}{2} \)^2}{N-1 + \(\gamma - \frac{N}{2}\)^2} & (\gamma \le 1), \\
	\(\gamma + \frac{N}{2}-1\)^2 + 2 + {\displaystyle\min_{\kappa \, \ge 0}} \(\kappa + \frac{4(N-1)(\gamma-1)}{\kappa + N-1 + \(\gamma - \frac{N}{2} \)^2} \) & (N \ge 4, \gamma > 1), \\
	\( \gamma + \frac{1}{2} \)^2 + 2 & (N = 3, \gamma > 1),
	\end{cases}
\]
}
\vspace{1em}

Note that the expression of the best constant $C(N, \gamma)$ is slightly different from that in \cite{Costin-Mazya} when $N \ge 4$,
but a careful checking the proof in \cite{Costin-Mazya} leads to the above formula in Theorem A.
Choosing $\gamma = 0$ in Theorem A,
we see that the best constant in \eqref{Hardy} is actually improved for axisymmetric divergence-free vector fields
in the sense that 
\[
	C_{N,0} \intRN \frac{|\bu|^2}{|\bx|^2} d{x} \le \intRN |\nabla \bu|^2 d{x}
\]
holds with the optimal constant $C_{N,0}= \(\frac{N}{2}-1\)^2 \frac{N^2+4N+4}{N^2+4N-4} > \(\frac{N-2}{2}\)^2$.

In 2-dimensional case, the result in \cite{Costin-Mazya} reads as follows:

%
%

\vspace{1em}\noindent
{\bf Theorem B.} (Costin-Maz'ya \cite{Costin-Mazya})
{\it
Let $\gamma \neq 0$ be a real number and $\bu \in C^\infty_c(\re^2)^2$ be a divergence-free vector field. 
We assume that $\bu({\bm 0})={\bm 0}$ if $\gamma < 0$ . 
Then
\[
	C_{2,\gamma} \int_{\re^2} \frac{|\bu|^2}{|\bx|^2} |\bx|^{2\gamma} d{x} \le \int_{\re^2}|\nabla \bu|^2 |\bx|^{2\gamma} d{x}
\]
holds with the optimal constant $C_{2,\gamma}$ given by
\[  
	C_{2,\gamma} = 
	\begin{cases}
	\gamma^2 \dfrac{3 + \(\gamma - 1\)^2}{1 + \(\gamma - 1\)^2} & \text{for} \quad |\gamma + 1| \le \sqrt{3}, \\
	\gamma^2 + 1 & \text{otherwise}.
	\end{cases}
\]
}
\vspace{1em}

When $N = 2$, $\bu$ in Theorem B need not be axisymmetric. 
Furthermore if we consider $\bu^{\perp} = (-u_2, u_1)$ for $\bu = (u_1, u_2)$ in Theorem B, 
then the condition ${\rm div} \, \bu = 0$ is replaced by ${\rm curl} \, \bu^{\perp} = 0$ and also $|\nabla \bu|^2 = |\nabla \bu^{\perp}|^2$. 
Thus the above inequality in Theorem B holds also for curl-free vector fields with the same constant.

Motivated by this observation,
our aim in this paper is to generalize Costin-Maz'ya's result for curl-free vector fields when $N=2$ to higher-dimensional cases.
In addition, we also consider the Rellich type inequality involving higher-order derivative, $\Delta \bu$, for curl-free vector fields.
We refer to \cite{Hamamoto} for the Rellich-Leray inequality for divergence-free vector fields. 

%
%
Now, main results of this paper are as follows:

%
%

\begin{theorem}
\label{Theorem:Hardy-Leray}
Let $\gamma \neq 1-N/2$ be a real number and let $\bu \in C_c^{\infty}(\re^N)^N$ be a curl-free vector field.
We assume that $\bu({\bm 0}) = {\bm 0}$ if $\gamma < 1-N/2$.
Then
\begin{equation}
\label{Eq:Hardy-Leray}
	H_{N,\gamma} \intRN  \frac{|\bu|^2}{|\bx|^2} |\bx|^{2\gamma} d{x} \le \intRN  |\nabla \bu|^2 |\bx|^{2\gamma} d{x}
\end{equation}
with the optimal constant $H_{N,\gamma}$ given by 
\begin{equation}
\label{H_N}
	H_{N,\gamma} = 
	\begin{cases}
	\(\gamma + \frac{N}{2}-1\)^2 \frac{3(N-1) + \(\gamma + \frac{N}{2}-2 \)^2}{N-1 + \(\gamma+\frac{N}{2}-2\)^2} & \text{if} \quad \big| \gamma + \frac{N}{2} \big| \le \sqrt{N+1}, \\ 
	\(\gamma + \frac{N}{2}-1\)^2 + N-1 & \text{otherwise}.
	\end{cases}
\end{equation}
\end{theorem}
We remark that no symmetry assumption for $\bu$ is needed.
Theorem \ref{Theorem:Hardy-Leray} corresponds to the higher-dimensional analogue of Theorem B in the sense that $C_{2,\gamma} = H_{2,\gamma}$.

For curl-free vector fields $\bu$, Poincar\'e's lemma implies that there exists a smooth $\phi$ such that $\bu = \nabla \phi$. 
Thus by using the potential function $\phi$, Theorem \ref{Theorem:Hardy-Leray} is equivalent to the following corollary.

%
%
\begin{corollary}
\label{Corollary:Hardy-Leray}
Let $\gamma \neq 1-N/2$ be a real number and let $\phi \in C_c^{\infty}(\re^N)$.
We assume that $\nabla \phi({\bm 0}) = {\bm 0}$ if $\gamma < 1-N/2$. 
Then
\[
	H_{N,\gamma} \intRN  \frac{|\nabla \phi|^2}{|\bx|^2} |\bx|^{2\gamma} d{x} \le \intRN  |D^2 \phi|^2 |\bx|^{2\gamma} d{x}
\]
with the optimal constant $H_{N,\gamma}$ given in \eqref{H_N}.
Here $D^2 \phi(\bx) = \( \frac{\pd^2 \phi}{\pd x_i \pd x_j}(\bx) \)_{1 \le i,j \le N}$ denotes the Hessian matrix of $\phi$.
\end{corollary}

By similar arguments, we prove the following Rellich-Leray inequality for curl-free vector fields.

%
%
\begin{theorem}
\label{Theorem:Rellich-Leray}
Let $\gamma \neq 2 - N/2$ be a real number and let $\bu \in C_c^{\infty}(\re^N)^N$ be a curl-free vector field.
We assume that $\intRN  |\bx|^{2\gamma-4}|\bu|^2 d{x} < \infty$.
Then
\begin{equation}
\label{Eq:Rellich-Leray}
	R_{N,\gamma} \intRN  \frac{|\bu|^2}{|\bx|^4} |\bx|^{2\gamma} d{x} \le \intRN  |\Delta \bu|^2 |\bx|^{2\gamma} d{x} 
\end{equation}
with the optimal constant $R_{N,\gamma}$  given by
\begin{equation}
\label{R_N}
	R_{N,\gamma} = \min_{\nu \in \N \cup \{ 0 \}} 
	\frac{(1 - \frac{N}{2} -\gamma)^2 + \al_\nu}{\( 3 - \frac{N}{2} - \gamma \)^2 + \al_\nu} (\al_{3-\frac{N}{2}-\gamma} -\al_\nu)^2
\end{equation}
for $\gamma \ne 3 - N/2$, 
where we put
\[
	\al_s = s(s+N-2), \quad s \in \re,
\]
%
%
and
\begin{equation}
\label{R_N_0}
	R_{N, 3 - N/2}= 
	\begin{cases} 
	4(N-2)^2 \quad & \text{for} \quad N=2,3,4, \\
	(N+3)(N-1) \quad & \text{for} \quad N \ge 5.
	\end{cases}
\end{equation}
\end{theorem}

%
%
\begin{corollary}
\label{Corollary:Rellich-Leray}
Let $\gamma \neq 2 - N/2$ be a real number and let $\phi \in C_c^{\infty}(\re^N)$ be a potential function such that
$\intRN  |\bx|^{2\gamma-4}|\nabla \phi|^2 d{x} < \infty$.
Then
\[
	R_{N,\gamma} \intRN \frac{|\nabla \phi|^2}{|\bx|^4} |\bx|^{2\gamma} d{x} \le \intRN |\nabla \Delta \phi|^2 |\bx|^{2\gamma} d{x} 
\]
holds with the optimal constant $R_{N,\gamma}$ as in \eqref{R_N} and \eqref{R_N_0}.
\end{corollary}

\begin{remark}
We do not know that the optimal constants $H_{N, \gamma}$ and $R_{N, \gamma}$ are attained or not
in the class of vector fields in Theorem \ref{Theorem:Hardy-Leray} and Theorem \ref{Theorem:Rellich-Leray}.

Also in Theorem \ref{Theorem:Hardy-Leray} and Theorem \ref{Theorem:Rellich-Leray}, if we restrict ourselves only on vector fields in $C_c^{\infty}(\re^N \setminus \{{\bm 0}\})$, then
the additional assumptions $\bu({\bm 0}) = {\bm 0}$ if $\gamma < 1-N/2$, or $\intRN  |\bx|^{2\gamma-4}|\bu|^2 d{x} < \infty$ are not needed
and the same conclusions hold true.
\end{remark}

Concerning Corollary \ref{Corollary:Hardy-Leray} which is equivalent to Theorem \ref{Theorem:Hardy-Leray},
we should remark that the similar results already exist by \cite{T-Z}, \cite{Ghoussoub-Moradifam}; see also \cite{Ghoussoub-Moradifam(book)} Chapter 6.5. 
More precisely,
improving the work by Tertikas and Zographopoulos \cite{T-Z}, 
Ghoussoub and Moradifam (\cite{Ghoussoub-Moradifam}: Appendix B) prove the following:
Let $C_c^{\infty}(B_R)$ denote the set of smooth functions having compact supports in a ball $B_R \subset \re^N$ with radius $R$. 
Define
\[
	A_{N, \gamma}(R) = \inf \left\{ \dfrac{\int_{B_R} |\Delta \phi|^2 |\bx|^{2\gamma} d{x}}{\int_{B_R} \frac{|\nabla \phi|^2}{|\bx|^2} |\bx|^{2\gamma} d{x}} \, : \, \phi \in C_c^{\infty}(B_R) \right\}.
\]
Assume $\gamma \ge 1 - N/2$.
Then $A_{N, \gamma}(R)$ is independent of $R$, and is equal to
\[
	A_{N, \gamma} = \min_{\nu \in \N \cup \{ 0 \}} \left\{ \dfrac{\( \frac{(N-4+2\gamma)(N-2\gamma)}{4} + \al_\nu \)^2}{\( \frac{N-4+2\gamma}{2}\)^2 + \al_\nu} \right \},
\]
where $\al_{\nu} = \nu(N + \nu - 2)$ $(\nu \in \N \cup \{ 0 \})$ is the $\nu$-th eigenvalue of the Laplace-Beltrami operator on the unit sphere $\Sp^{N-1}$ in $\re^N$.
Note that by a simple formula
\[
	|D^2 \phi|^2 = \sum_{i,j=1}^N \(\frac{\pd^2 \phi}{\pd x_i \pd x_j} \)^2 
	= {\rm div} \, \( \frac{1}{2} \nabla |\nabla \phi|^2 - (\Delta \phi) \nabla \phi \) + (\Delta \phi)^2,
\]
for $\phi \in C_c^{\infty}(B_R)$, 
we have $\int_{B_R} |D^2 \phi|^2 d{x} = \int_{B_R} |\Delta \phi|^2 d{x}$ which implies $H_{N, 0} = A_{N, 0}$. 
However, in weighted cases, it holds $\int_{B_R} |D^2 \phi|^2 |\bx|^{2\gamma} d{x} \ne \int_{B_R} |\Delta \phi|^2 |\bx|^{2\gamma} d{x}$, 
and in general we have $H_{N, \gamma} \ne A_{N, \gamma}$.
Also the inequality in Corollary \ref{Corollary:Rellich-Leray} seems new.

The organization of this paper is as follows:
In \S 2, we recall the method by Costin-Maz'ya in \cite{Costin-Mazya} and derive the equivalent curl-free condition in polar coordinates.
In \S 3, we prove Theorem \ref{Theorem:Hardy-Leray} and the sharpness of the constant \eqref{H_N}.
In \S 4, we prove Theorem \ref{Theorem:Rellich-Leray} and the sharpness of the constants \eqref{R_N} and \eqref{R_N_0}.
Since the test vector fields introduced in \cite{Costin-Mazya} may not have compact supports, 
we will use different test vector fields for the proof of the sharpness of the constants.

%
%
\section{Preparation: Costin-Maz'ya's setting}

In this section, we recall the method of Costin-Maz'ya \cite{Costin-Mazya} and derive the polar coordinate representation of the curl-free condition.

\vspace{1em}\noindent
\subsection*{Spherical polar coordinate}

We introduce the spherical polar coordinates
\[
	(\rho,\th_1,\th_2,\cdots,\th_{N-2,}\th_{N-1}) \in (0,\infty) \times (0,\pi)^{N-2} \times [0,2\pi)
\]
whose relation to the standard Euclidean coordinates $\bx=(x_1,\cdots,x_N)\in\mathbb{R}^N$ is given by
\[
	\bx = \rho (\cos\th_1, \pi_1 \cos\th_2, \pi_2 \cos\th_3, \cdots, \pi_{k-1} \cos\th_k, \cdots, \pi_{N-2} \cos\th_{N-1}, \pi_{N-1}),
\]
hereafter we use the notation
\[
\pi_0 = 1, \quad  \pi_k = \prod_{j=1}^k \sin \th_j, \quad (k=1,2, \cdots, N-1)
\]
for simplicity. 
Also we use the notation
\[
	\pd_\rho = \frac{\pd}{\pd\rho}, 
	\quad \pd_{\th_k}=\frac{\pd }{\pd{\th_k}}, \quad (k=1,2, \cdots, N-1)
\]
for the partial derivatives, and 
\begin{align*}
	d{x} =\prod_{k=1}^{N}dx_k=dx_1dx_2\cdots dx_N, \quad
	d\sigma =\prod_{k=1}^{N-1} (\sin\th_k)^{N-k-1}d\th_k
\end{align*}
for the volume elements on $\re^N$ and $\Sp^{N-1}$.

The orthonormal basis vector fields $\be_\rho, \be_{\th_1}, \be_{\th_2}, \cdots, \be_{\th_{N-1}}$ along the polar coordinates are given by
\begin{equation}
\label{ONB}
	\begin{cases}
	&\be_\rho = \frac{\pd_\rho \bx}{|\pd_\rho \bx|}  
	=\(\cos\th_1, \pi_1 \cos\th_2, \pi_2 \cos\th_3, \cdots \pi_{N-2} \cos\th_{N-1}, \pi_{N-1} \),\vspace{0.2em} \\
	&\be_{\th_k} = \frac{\pd_{\th_k} \bx}{|\pd_{\th_k} \bx|} 
	= \dfrac{1}{\pi_{k-1}\hspace{-1em}}\,\pd_{\th_k}\be_{\rho},  \quad (k=1,2,\cdots,N-1)
	\end{cases}
\end{equation}
that are clearly independent of the variable $\rho$.
Note that we can rewrite
\begin{align*}
	&\be_\rho = (\cos\th_1, \,\pi_1 \cos\th_2\,, \,\pi_2 \cos\th_3\,, \cdots ,\pi_{k-1} \cos\th_k\,, \pi_k \bphi_k)\ , \\
	&\be_{\th_k} = \big(\underbrace{\strut 0,0,\cdots,0}_{k-1}, -\sin\th_k,\, \cos\th_k \bphi_k\big)\ ,
\end{align*}
where
\begin{align*}
	\bphi_k = \(\cos\th_{k+1}, \frac{\pi_{k+1}}{\pi_k} \cos\th_{k+2}, \frac{\pi_{k+2}}{\pi_k} \cos\th_{k+3}, \cdots, \frac{\pi_{N-2}}{\pi_k} \cos\th_{N-1}, \frac{\pi_{N-1}}{\pi_k}\) \in \Sp^{N-k-1}
\end{align*}
is a $(N-k)$-vector, which depends only on $\th_{k+1}, \cdots, \th_{N-1}$.
From these expressions, we can easily check the orthonormality of $\be_\rho,\be_{\th_1},\be_{\th_2},\cdots,\be_{\th_{N-1}}$.

For any smooth vector field $\bu = (u_1,u_2,\cdots,u_N):\re^N \to \re^N$,
its polar components $u_\rho$, $u_{\th_1}$, $u_{\th_2}$, $\cdots$, $u_{\th_{N-1}}$ as $\re$-valued smooth functions are defined by
\[
	 \bu = u_\rho \be_\rho + \sum_{k=1}^{N-1} u_{\th_k}\be_{\th_k}\ .
\]
The second term of the right-hand side is denoted as
\[
	 \bu_{\bsig}=\sum_{k=1}^{N-1}u_{\th_k}\be_{\th_k}
\]
and we call this the spherical component of $\bu$.   
Thus we have the polar decomposition of $\bu$:
\begin{equation}
\label{polar-components}
	\bu =u_\rho\be_\rho + \bu_{\bsig}
\end{equation}
which gives the decomposition of $\bu$ into the radial and the spherical parts.
Also by using the chain rules together with \eqref{ONB}, we have 
\[
	 \pd_\rho 
	=\be_\rho \cdot \nabla, \quad \ {\rm and}\ \quad \frac{1\,}{\rho\,}\hspace{0.1em} \pd_{\th_k} 
	=\pi_{k-1} \be_{\th_k}\cdot \nabla, \quad(k=1,\cdots,N-1),
\]
which give the polar decomposition of the gradient operator $\nabla$:
\begin{equation}
\label{polar-nabla}
	 \nabla =  \be_\rho \pd_\rho + \frac{1\,}{\rho\,}\nabla_{\bsig}\ ,
\end{equation}
where
\begin{equation}
\label{spherical_gradient}
	 \nabla_{\bsig}= \sum_{k=1}^{N-1}\frac{\,\be_{\th_k}}{\pi_{k-1}\hspace{-0.7em}}\,\pd_{\th_k}
\end{equation}
is the gradient operator on $\Sp^{N-1}$.

Moreover, it is well-known that the polar representation of the Laplace operator $\Delta = \sum_{k=1}^N \pd^2/\pd x_k^2$ is given by
\begin{equation}
 \label{polar-Delta}
	 \Delta = \frac{1}{\rho^{N-1}} \pd_\rho \( \rho^{N-1} \pd_\rho \) + \frac{1}{\rho^2} \Delta_{\bsig} \ ,
\end{equation}
where
\begin{equation}
 \label{spherical_Laplacian}
	\Delta_{\bsig} = \sum_{k=1}^{N-1} \frac{(\sin\th_k)^{\strut k+1-N}\hspace{-2.7em}}{\pi_{k-1}^2}\ \pd_{\th_k} \((\sin\th_k)^{N-k-1} \pd_{\th_k}\)
	= \sum_{k=1}^{N-1} \frac{1}{\pi_{k-1}^2\hspace{-0.7em}} \,D_{\th_k} \pd_{\th_k}
\end{equation}
is the Laplace-Beltrami operator on $\Sp^{N-1}$ and for every $k=1,\cdots,N-1$
\[
	D_{\th_k} = \pd_{\th_k} + (N-k-1) \cot \th_k
\] 
is the adjoint operator of $-\pd_{\th_k}$ in $L^2(d\sigma, \Sp^{N-1})$ : it holds that
\[
	-\int_{\Sp^{N-1}} f \(\pd_{\th_k} g\) d\sigma = \int_{\Sp^{N-1}} \(D_{\th_k} f\) g d\sigma
\]
for any smooth functions $f, g$ on $\Sp^{N-1}$.

We also introduce the deformed radial coordinate $t \in \re$ by the Emden transformation
\begin{equation}
\label{Emden}
	t = \log \rho.
\end{equation}
Note that \eqref{Emden} leads to the transformation law of the differential operators $\rho \pd_\rho = \pd_t$.
By this transformation, it is easy to check that the polar decomposition of $\nabla$ , $\Delta$ in \eqref{polar-nabla} , \eqref{polar-Delta} are also given by
\begin{align}
\label{Emden-nabla}
	&\rho\nabla={\bm e}_\rho \partial_t+\nabla_{\bsig}, \\
\label{Emden-Delta}
	&\rho^2 \Delta = \pd_t^2 + (N-2) \pd_t + \Delta_{\bsig}.
\end{align}
For the later use, we prove the following lemma.

%
%
\begin{lemma}
\label{Lemma:commutation}
Let $\nabla_{\bsig}$ and $\Delta_{\bsig}$ are defined by \eqref{spherical_gradient} and \eqref{spherical_Laplacian} respectively.
Then for any $f \in C^{\infty}(\mathbb{S}^{N-1})$ , \ $\bsig=\be_\rho\in\mathbb{S}^{N-1}$ and $\al \in \C$, there hold that
\begin{align}
\label{commutation1}
	\Delta_{\bsig} (\be_{\rho} f) - \be_{\rho} \Delta_{\bsig} f &= \big( 2 \nabla_{\bsig} - (N-1) \be_{\rho} \big) f, \\
\label{commutation2}
	\Delta_{\bsig} \nabla_{\bsig} f - \nabla_{\bsig} \Delta_{\bsig} f &= \big( (N-3) \nabla_{\bsig} - 2 \be_{\rho} \Delta_{\bsig} \big) f, \\
\label{commutation3}
	\Delta_{\bsig} \(f {\bm e}_\rho + \al \nabla_{\bsig}f \) &= {\bm e}_\rho \big((1-2\al) \Delta_{\bsig}f - (N-1) f \big) \\
	&\quad  + \big( 2 + (N-3)\al \big) \nabla_{\bsig}f + \al \nabla_{\bsig} \Delta_{\bsig}f. \notag
\end{align}
\end{lemma}

\begin{proof}
 Take any $f \in C^{\infty}(\mathbb{S}^{N-1})$. We identify $f$ and the function $\widetilde{f}\in C^\infty(\mathbb{R}^N\backslash\{\bm 0\})$ defined by $\widetilde{f}(\bx)=f(\bx/|\bx|)$ . Since $f=\widetilde{f}$ does not depend on the radial variable $\rho$, 
we have $\nabla_{\bsig}f=\rho\nabla f$  by \eqref{polar-nabla} and $\Delta_{\bsig}f=\rho^2\Delta f$ by \eqref{polar-Delta}.  Thus we compute
\begin{align*}
	\Delta_{\bsig}(\be_\rho f) - \be_\rho \Delta_{\bsig}f 
 &= \rho^2 \Delta \( \frac{\bx f}{\rho} \) - \frac{\bx\hspace{0.1em}}{\rho}\hspace{0.1em} \rho^2 \Delta f \\
	&= \rho^2 \( \frac{\Delta(\bx f)}{\rho}+2\(\(\nabla \rho^{-1}\) \cdot \nabla \)(\bx f)+(\Delta\rho^{-1})\bx f \)-\rho \bx \Delta f \\
	&= 2\rho(\nabla f\cdot \nabla)\bx - 2\(\nabla \rho \cdot \nabla\)(\bx f)+\rho^3 (\Delta \rho^{-1}) \be_\rho f \\
	& =2\rho \nabla f-2\pd_\rho(\rho \be_\rho f) - (N-3) \be_\rho f \\
	& = \big(2\nabla_{\bsig}-(N-1) \be_\rho \big)f, 
\end{align*}
here we have used $\nabla \rho \cdot \nabla = \pd_{\rho}$ and $\Delta \rho^{-1} = -(N-3) \rho^{-3}$.
This proves \eqref{commutation1}.
 Similarly, also noting the commutativity $\Delta\nabla=\nabla\Delta$ and using $\Delta\rho=(N-1)\rho^{-1}$, we have
\begin{align*}
	(\Delta_{\bsig} \nabla_{\bsig} - \nabla_{\bsig} \Delta_{\bsig} )f
 &= \rho^{2}\Delta\nabla_{\bsig} f - \rho\nabla \Delta_{\bsig} f
\\&=\rho^2\Delta\big(\rho\nabla f\big)-\rho\nabla\big(\rho^2\Delta f\big)
 \\&=\rho^2\big((\Delta\rho)\nabla f+2(\nabla\rho\cdot\nabla)\nabla f\big)-\rho\big(\nabla\rho^2\big)\Delta f
 \\&=(N-1)\rho\nabla f+2\rho^2\partial_\rho\rho^{-1} \nabla_{\bsig}f-2\rho^2 \be_\rho \Delta f
\\&= (N-3) \nabla_{\bsig}f - 2\be_\rho \Delta_{\bsig}f \ .
\end{align*}
This proves \eqref{commutation2}.
Finally, by \eqref{commutation1} and \eqref{commutation2}, we see
\begin{align*}
	\Delta_{\bsig} \(f {\bm e}_\rho + \al \nabla_{\bsig}f \) &= \Delta_{\bsig}({\bm e}_\rho f ) + \al \Delta_{\bsig} \nabla_{\bsig}f \\
	&= \({\bm e}_\rho \Delta_{\bsig} + 2\nabla_{\bsig} - (N-1){\bm e}_\rho \) f + \al \( \nabla_{\bsig} \Delta_{\bsig} + (N-3) \nabla_{\bsig} - 2 {\bm e}_\rho \Delta_{\bsig} \) f \\
	&= {\bm e}_\rho \((1-2\al) \Delta_{\bsig}f - (N-1) f \) + \( 2 + (N-3)\al \) \nabla_{\bsig}f + \al \nabla_{\bsig} \Delta_{\bsig}f.
\end{align*}
This proves \eqref{commutation3}.
\end{proof}

%
%
\vspace{1em}\noindent
\subsection*{Representing the curl-free condition in polar coordinates}

In the following, let ``$\cdot$" denote the standard inner product in $\re^N$,
``$\wedge$" the wedge product for differential forms and ``$d$" the exterior derivative operator.
For a vector field ${\bm a} = (a_1, a_2, \cdots, a_N): \re^N \to \re^N$, we define the vector valued 1-form 
$d{\bm a} = (da_1, da_2, \cdots, da_N)$.
If $\bu = (u_1, u_2, \cdots, u_N)$ is a vector field, 
then $\bu \cdot d{\bm a}$ denotes the 1-form $\sum_{k=1}^N u_k da_k$.
Now, for any $C^\infty$ vector field $\bu: \re^N \to \re^N$ with variable ${\bm x}=(x_1,\cdots,x_N)$, 
we define its curl as the differential $2$-form
\[
	{\rm curl}\,\bu = d(\bu \cdot d\bx).
\] 
This can be expressed in terms of the standard Euclidean coordinates, according to the usual manipulations for the exterior derivative $d$ and the wedge product $\wedge$ :
\[
	d(\bu \cdot d\bx) = \sum_{k=1}^N du_k \wedge dx_k
	=\sum_{k=1}^{N}\sum_{j=1}^N \frac{\partial u_k}{\partial x_j}dx_j\wedge dx_k=\underset{j<k}{\sum\sum}\(\frac{\partial u_k}{\partial x_j}-\frac{\partial u_j}{\partial x_k}\)dx_j\wedge dx_k  .
\]
As well as the standard representation, we want to find a representation of $d(\bu\cdot d\bx)$ in terms of the polar coordinates $(\rho,\theta_1,\cdots,\theta_{N-1})$.
For this purpose, first we differentiate the unit vector field $\be_\rho$ given by \eqref{ONB} and expand it in the spherical-coordinate basis:
\[ 
	d \be_\rho = \sum_{k=1}^{N-1}\frac{\partial {\bm e}_\rho}{\partial \theta_k}d\theta_k=\sum_{k-1}^{N-1}  \be_{\th_k}\pi_{k-1} d\th_k\ .
\]
Then, taking the inner product with the vector field ${\bm u}=u_\rho {\bm e}_\rho+\sum_{k=1}^{N-1}u_{\theta_k}{\bm e}_{\theta_k}$ 
and also taking its exterior derivative,
we see that
\[
\begin{split}
	{\bm u}\cdot d\be_\rho &=\sum_{k=1}^{N-1}u_{\theta_k}\pi_{k-1}d\theta_k\ , \\  
	d({\bm u}\cdot d {\bm e}_\rho)&= d\rho\wedge\sum_{k=1}^{N-1}(\partial_\rho u_{\theta_k})\pi_{k-1}d\theta_k+\sum_{j=1}^{N-1}\sum_{k=1}^{N-1}\partial_{\theta_j}(\pi_{k-1}u_{\theta_k})d\theta_j\wedge d\theta_k\ .
\end{split}
\]
Also we have
\[
	\bu \cdot d {\bm x} ={\bm u}\cdot \Big((d\rho) {\bm e}_\rho+\rho \hspace{0.1em}d {\bm e}_\rho\Big) = u_\rho d\rho+\rho {\bm u}\cdot d {\bm e}_\rho.
\]
From these relations, we obtain the polar representation of the curl of $\bu$:
\[
\begin{split} 
	d({\bm u}\cdot d {\bm x}) 
	&=d(u_\rho d\rho+\rho {\bm u}\cdot d {\bm e}_\rho) \\
	&= du_{\rho} \wedge d\rho + d\rho \wedge (\bu \cdot d\be_{\rho}) + \rho d(\bu \cdot d\be_{\rho})
	\\&=d\rho\wedge\Big(-du_\rho+\sum_{k}u_{\theta_k}\pi_{k-1}d\theta_k+\sum_{k}\rho\partial_\rho u_{\theta_k}\pi_{k-1}d\theta_k\Big)
	\\&\quad+\rho\sum_{j}\sum_{k}\partial_{\theta_j}\big(\pi_{k-1}u_{\theta_k}\big)d\theta_j\wedge d\theta_k
	\\&=d\rho\wedge\sum_{k}\Big(\pi_{k-1}\partial_\rho(\rho u_{\theta_k})-\partial_{\theta_k}u_\rho\Big) d\theta_k
	\\&\quad +\rho\underset{j<k}{\sum\!\sum}\Big(\partial_{\theta_j}\big(\pi_{k-1}u_{\theta_k}\big)-\partial_{\theta_k}\big(\pi_{j-1}u_{\theta_j}\big)\Big)d\theta_j\wedge d\theta_k\ .
\end{split}
\]
Therefore, the curl-free condition $d(\bu \cdot \bx)=0$ for the vector field $\bu$ is represented  by 
\begin{equation}
\label{curl-free(polar)}
	\begin{cases}
	&\pd_\rho\big(\rho\hspace{0.05em} \pi_{k-1}u_{\th_k}\big)=\pd_{\th_k} u_\rho\   \\
	&\partial_{\theta_j}\big(\pi_{k-1}u_{\theta_k}\big)=\partial_{\theta_k}\big(\pi_{j-1}u_{\theta_j}\big)\ 
	\end{cases}
	\ , \quad \big(j, k = 1, 2, \cdots, N-1 \ \big).
\end{equation}
We claim that the second relations in \eqref{curl-free(polar)} are actually the consequences of the first.
Indeed, by integrating the first equation in \eqref{curl-free(polar)} on any interval $(0,r]\subset\mathbb{R}$ with respect to the variable $\rho$, 
we have $r\pi_{k-1} u_{\theta_k}=\partial_{\theta_k}\int_{0}^{r}u_{\rho} d{\rho}$ for every $k$. 
Thus the function $\phi\in C^\infty\big(\mathbb{R}^N\backslash\{\bm 0\}\big)$ defined by $\phi(\bx)=\frac{1}{|\bx|}\int_{0}^{|\bx|}u_\rho\big(\rho\bx/|\bx|\big) d\rho$ satisfies 
$\pi_{k-1}u_{\theta_k}=\partial_{\theta_k}\phi$ for all $k$. 
Then the second relation in \eqref{curl-free(polar)} is the same as $\partial_{\theta_j}\partial_{\theta_k}\phi=\partial_{\theta_k}\partial_{\theta_j}\phi$, which holds trivially.
This proves the claim.

Consequently we have proved that a vector field ${\bm u}\in C^\infty(\mathbb{R}^N)^N$ is curl-free if and only if
\[ 
	\partial_{\rho}(\rho u_{\theta_k})=\frac{1}{\pi_{k-1}\hspace{-1em}}\ \partial_{\theta_k}u_\rho \ , \qquad k=1,\cdots,N-1\ .
\]
That is, using vector notation as in \eqref{polar-components} and \eqref{polar-nabla}, we have
\begin{equation}
\label{curl-free}
	\partial_{\rho}(\rho \bu_{\bsig}) = \nabla_{\bsig}u_\rho\ \quad (\rho,\bsig) \in\mathbb{R}_+\times\mathbb{S}^{N-1}.
\end{equation}
In what follows we also call \eqref{curl-free} the curl-free condition for $\bu$.

\vspace{1em}\noindent
\subsection*{Brezis-Vazquez, Maz'ya transformation}

Let $\eps \ne 1$ be a real number.
As in \cite{Brezis-Vazquez}, \cite{Mazya}, 
we introduce a new vector field $\bv$ by the formula
\begin{equation}
\label{BV}
	  \bv(\bx) = \rho^{1-\eps} \bu(\bx).
\end{equation}
Then the curl-free condition \eqref{curl-free} is transformed into
\[
	\nabla_{\bsig}\big(\rho^{\varepsilon-1}v_\rho\big)=\partial_\rho(\rho^\varepsilon {\bm v}_{\bsig})\ ,
\]
that is,
\begin{equation}
\label{curl-free-v-rho}
	(\varepsilon+\rho\partial_\rho){\bm v}_{\bsig}=\nabla_{\bsig}v_\rho\ .
\end{equation}

\vspace{1em}\noindent
\subsection*{Fourier transformation in radial direction}

In the following, let us denote $\bx = \rho \bsig$ and we abbreviate $\bv(t, \bsig) = \bv(e^t \bsig)$ for a vector field $\bv(\bx) = \bv(\rho \bsig)$,
where $t=\log\rho$ is the Emden transformation given in \eqref{Emden}.
As in \cite{Costin-Mazya}, we apply the one-dimensional Fourier transformation
\[
	\bv(t,{\bsig}) \mapsto \wh{\bv}(\la,{\bsig}) = \frac{1}{\sqrt{2\pi}} \int_{\re} e^{-i\lambda t} \bv(t,{\bsig})dt
\]
with respect to the variable $t$.
By the transformation law between the derivative operators
\[
	\rho \pd_{\rho} = \pd_t\ \mapsto\  \wh{\pd_t} = i\la\, \cdot\ ,
\]
the curl-free condition \eqref{curl-free-v-rho} is changed into the equation
\begin{align*}
\label{curl-free-vhat(component)}
	(\varepsilon+i\lambda)\widehat{\bv_{\bsig}}=\nabla_{\bsig}\widehat{\hspace{0.1em} v_\rho}\ ,
\end{align*}
that is,
\[
	\widehat{\bv_{\bsig}}=\frac{1}{\varepsilon+i\lambda}\nabla_{\bsig}f\qquad  \text{where} \quad f=\widehat{\hspace{0.1em}v_\rho}\ .
\]
Thus we see that $\wh{\bv}_{\bsig}$ is expressed by the spherical gradient of some function $f=\wh{\hspace{0.1em}v_\rho}$.
In this sense, we may consider $f$ as a kind of scalar potential of $\widehat{\bv}$, 
corresponding to the fact that the curl-free vector field ${\bm u}$ has a scalar potential. 

Now we have proved the following proposition:

%
%
\begin{proposition}
\label{Proposition:curl-free}
Let $\eps \ne 1$ and let $\bu$ be a smooth vector field on $\re^N$. 
Then $\bu$ is curl-free if and only if its Brezis-V\'azquez, Maz'ya transformation $\bv(t, \bsig) = e^{t(1-\eps)} \bu(e^t \bsig)$ satisfies
\begin{equation}
\label{curl-free-v(component)}
	(\varepsilon+\partial_t){\bm v}_{\bsig}=\nabla_{\bsig}v_\rho\ .
\end{equation}
In particular, if $\bu$ is curl-free and has a compact support on $\mathbb{R}^N$, then the Fourier transformation of $\bv$ satisfies
\begin{equation}
\label{curl-free-vhat2}
	\wh{\bv}(\la, \bsig) = f \be_{\rho} +  \frac{1}{\eps + i\lambda} \nabla_{\bsig} f
\end{equation}
for some complex-valued scalar function $f = f(\la, \bsig) \in C^{\infty}(\mathbb{R}\times\mathbb{S}^{N-1})$.
\end{proposition}

We list up some formulae for $\wh{\bv}$ and its differentials. 
The square length of $\wh{\bv}$ is 
\[
	|\wh{\bv}|^2 = |f|^2 + \frac{1}{\eps^2+\la^2} |\nabla_{\bsig} f|^2.
\]
By using Lemma \ref{Lemma:commutation}, we also see that
 \[
	 -\Delta_{\bsig} \wh{\bv} = \be_\rho \((N-1)f+\( \frac{2}{\eps+i\la}-1\) \Delta_{\bsig} f\) 
	-\(\frac{N-3}{\eps+i\la}+2\) \nabla_{\bsig}f-\frac{1}{\eps+i\la}\nabla_{\bsig} \Delta_{\bsig}f\ .
 \]
Then integrating $|\wh{\bv}|^2$, $-\ol{\wh{\bv}} \cdot \Delta_{\bsig} \wh{\bv}$ and $|\Delta_{\bsig}\wh{\bv}|^2$
over $\Sp^{N-1}$,
we find that
\begin{align*}
	&\intS |\bvh|^2 d\sigma = \intS \ol{f} \(1+\frac{1}{\eps^2+\la^2}(-\Delta_{\bsig})\) f d\sigma, \\
	&\intS |\nabla_{\bsig} \bvh|^2 d\sigma = \intS \ol{f} 
	\(N-1+\(1+\frac{3-4\eps-N}{\eps^2+\la^2}\) (-\Delta_{\bsig}) + \frac{1}{\eps^2+\la^2}(-\Delta_{\bsig})^2 \)f d\sigma, \\
	&\intS |\Delta_{\bsig} \bvh|^2 d\sigma 
	= \intS \ol{f} \Bigg( (N-1)^2 + \(2N+2+\frac{(N-3)^2-8\eps}{\eps^2+\la^2}\)(-\Delta_{\bsig}) \\
	&\hspace{45mm}+\(1+\frac{10-8\eps-2N}{\eps^2+\la^2}\)(-\Delta_{\bsig})^2+\frac{1}{\eps^2+\la^2}(-\Delta_{\bsig})^3 \Bigg)f d\sigma.
\end{align*}

Thus, we have proved the following lemma.
%
%
\begin{lemma}
\label{Lemma:integrals}
Let $\wh{\bv} = f \be_\rho + \frac{1}{\eps + i\la} \nabla_{\bsig}f$ as in \eqref{curl-free-vhat2}. 
Then
 \begin{align*}
	&\intS |\bvh|^2 d\sigma = \intS \ol{f} P_1(\la, -\Delta_{\bsig})f d\sigma, \\
	&\intS |\nabla_{\bsig}\bvh|^2 d\sigma = \intS \ol{f} P_2(\la, -\Delta_{\bsig})f d\sigma, \\
	&\intS |\Delta_{\bsig}\bvh|^2 d\sigma = \intS \ol{f} P_3(\la, -\Delta_{\bsig})f d\sigma\ 
\end{align*}
where the three polynomials $\alpha \mapsto P_k(\la,\alpha)$ $(k=1,2,3)$ are given by
\begin{align*}
	&P_1(\la, \al) = 1 + \frac{1}{\eps^2+\la^2}\,\alpha, \\
	&P_2(\la, \al) = N-1+ \(1+\frac{3-4\eps-N}{\eps^2+\la^2}\)\al + \frac{1}{\eps^2+\la^2}\,\al^2, \\
	&P_3(\la, \al)= (N-1)^2 + \(2N+2+\frac{(N-3)^2-8\eps}{\eps^2+\la^2}\)\al \\
	&\hspace{30mm} + \(1+\frac{10-8\eps-2N}{\eps^2+\la^2}\) \al^2 + \frac{1}{\eps^2+\la^2}\,\al^3.
\end{align*}
\end{lemma}

%
%

\section{Proof of Theorem \ref{Theorem:Hardy-Leray}}

Let $\gamma \neq 1 - N/2$ be a real number and put $\eps = 2 - N/2 -\gamma \ne 1$.
If the right-hand side of \eqref{Eq:Hardy-Leray} diverges, there is nothing to prove.
When the right-hand side of \eqref{Eq:Hardy-Leray} is bounded, 
the smoothness of $\bu$ implies the existence of an integer $m > \eps -2$ such that $\nabla \bu(\bx) = O(|\bx|^m)$ as $|\bx| \to + 0$. 
If $\eps < 1$, then the Brezis-V\'{a}zquez, Maz'ya transformation $\bv(\bx)$ in \eqref{BV} is H\"{o}lder continuous at $\bx={\bm 0}$ and satisfies $\bv({\bm 0})={\bm 0}$. 
When $\eps > 1$, again the assumption $\bu({\bm 0})={\bm 0}$ implies $\bu(\bx) = O(|\bx|^{m+1})$ and $\bv(\bx) = O(|\bx|^{m+2-\eps})$, thus the same properties hold for $\bv$.
Also since 
\begin{align*}
	\nabla \bu &= \nabla (\rho^{\eps-1}\bv)
	= \rho^{\eps-2}\((\eps-1)\be_\rho \otimes \bv + \rho\nabla \bv \) \\ 
	&= \rho^{\eps-2} \(\be_\rho \otimes \(\eps-1+\pd_t \) \bv + \nabla_{\bsig}\bv \)
\end{align*}
by \eqref{Emden-nabla}, 
and since $\iint_{\Sp^{n-1} \times \re} \pd_t \bv \cdot \bv d\sigma dt$ vanishes,
we calculate
\begin{align}
\label{Hardy-RHS}
	\intRN |\bx|^{2\gamma} |\nabla \bu|^2 d{x} &= \intRN |\bx|^{4-2\eps-N} |\nabla \bu|^2 d{x} \\ 
	&= \intS  d\sigma \int_{\re} \big|\be_\rho \otimes \(\eps-1+\pd_t \) \bv + \nabla_{\bsig} \bv\big|^2 dt \notag \\
	&=\iint_{\Sp^{N-1} \times \re} \( (\eps-1)^2 |\bv|^2 + |\pd_t \bv|^2 + |\nabla_{\bsig} \bv|^2 \) d\sigma dt \notag \\
	&=\iint_{\Sp^{N-1} \times \re} \(\((\eps-1)^2+\la^2\) |\bvh|^2 + |\nabla_{\bsig}\bvh|^2 \) d\sigma d\la \notag \\
	&=\iint_{\Sp^{N-1} \times \re} \ol{f} \(\((\eps-1)^2+\la^2\)P_1(\la, -\Delta_{\bsig})+P_2(\la, -\Delta_{\bsig})\)f d\sigma d\la \notag 
\end{align}
and
\begin{align}
\label{Hardy-LHS}
	\intRN |\bx|^{2\gamma-2} |\bu|^2 d{x} &= \intRN |\bx|^{2-2\eps-N} |\bu|^2 d{x} \\
	&= \intS d\sigma \int_0^\infty |\bv|^2 \frac{d\rho}{\rho} = \iint_{\re \times \Sp^{N-1}} |\bv|^2 dt d\sigma \notag \\
	&= \iint_{\re \times \Sp^{N-1}} |\bvh|^2 d\la d\sigma = \iint_{\re \times \Sp^{N-1}} \ol{f} P_1(\la, -\Delta_{\bsig}) f d\la d\sigma \notag
\end{align}
by Lemma \ref{Lemma:integrals}. 
Therefore, by \eqref{Hardy-RHS} and \eqref{Hardy-LHS},
the optimal constant in \eqref{Eq:Hardy-Leray} can be expressed as
\begin{equation}
\label{H_N(New)}
	H_{N,\gamma} = \inf_{\bu \neq 0, {\rm curl} \, \bu = {\bm 0}} 
	\frac{\intRN  |\bx|^{2\gamma}|\nabla \bu|^2 d{x}}{\intRN |\bx|^{2\gamma-2}|\bu|^2 d{x}}
	= \inf_{f \neq 0} \frac{\iint_{\re \times \Sp^{N-1}}\ol{f} Q_1(\la,-\Delta_{\bsig})f d\la d\sigma}
	{\iint_{\re \times \Sp^{N-1}}\ol{f}P_1(\la, -\Delta_{\bsig})f d\la d\sigma}\ ,
\end{equation}
where $Q_1(\la,\,\cdot\,)$ is the polynomial defined by
\begin{equation}
\label{Q_1}
	Q_1(\la,\al) = \((\eps-1)^2 + \la^2 \) P_1(\la,\al) + P_2(\la,\al).
\end{equation}

%
%
\vspace{1em}\noindent
\subsection*{Calculation of a lower bound}

In the same manner as Costin-Maz'ya \cite{Costin-Mazya}, 
we expand $f$ in $L^2(\Sp^{N-1})$ by eigenfunctions $\{ \psi_\nu \}_{\nu \in \{0\} \cup \N}$ of $-\Delta_{\bsig}$ as
\begin{equation}
\label{expansion}
	f(\la,{\bsig}) = \sum_{\nu = 0}^\infty f_\nu(\la) \psi_\nu({\bsig})\ ,\ \quad
	\begin{cases}
		&-\Delta_{\bsig} \psi_\nu = \al_\nu \psi_\nu \ ,\\
		&\al_\nu = \nu (\nu + N - 2) \quad (\nu=0,1,2,\cdots).
	\end{cases}
\end{equation}
Then we find that \eqref{H_N(New)} is estimated from below as
\begin{align*}
	H_{N,\gamma} 
= \inf_{f \ne 0}  \frac{\sum_{\nu \in \N \cup \{ 0 \}} \int_{\re} Q_1(\la,\al_{\nu}) |f_{\nu}(\lambda)|^2 d\la}
	{\sum_{\nu \in \N \cup \{0\}} \int_{\re} P_1(\la,\al_{\nu}) |f_{\nu}(\lambda)|^2 d\la} \notag 
	\ge \inf_{\la \in \re\backslash\{0\}} \inf_{\nu \in \N \cup \{0\}} \frac{Q_1(\la,\al_\nu)}{P_1(\la,\al_\nu)} \ ,
\end{align*}
where $P_1$, $Q_1$ is as in Lemma \ref{Lemma:integrals}, \eqref{Q_1} and where in the last inequality 
 we have used Lemma \ref{Lemma:appendix} in Appendix, 
applied to $X = \left\{ (\nu, \la) \in \( \N \cup \{0\} \) \times \re \right \}$, $\mu = \( \sum_{\nu \in \N \cup \{ 0 \}} \delta_{\nu} \) \times d\la$ 
and $g(\nu, \la) = |f_{\nu}(\la)|^2$. Therefore, we have
\begin{equation}
\label{lbound of Hardy}
	H_{N,\gamma} \ge  \inf_{\kappa >0} \inf_{\nu \in \N \cup \{0\}} F(\kappa,\al_\nu)
\end{equation}
with $F(\kappa,\,\cdot\,)$ defined by
\begin{equation}
\label{F(Hardy)}
	F(\kappa,\al) = \frac{Q_1(\sqrt{\kappa},\al)}{P_1(\sqrt{\kappa},\al)}  = (\eps-1)^2 + N-1+\kappa + \al - 2\al \frac{2\eps+N-2}{\eps^2+\kappa+\al}
\end{equation}
for $\kappa>0$ and $\alpha\ge0$. 
Here we also define $F(0,\alpha)$ by
\begin{equation}
\label{F(Hardy0)}				  
\begin{split}
	F(0,\alpha)&=\lim_{|\lambda|\searrow +0}\frac{Q_1(\lambda,\alpha)}{P_1(\lambda,\alpha)}
	=\lim_{\kappa\searrow +0} F(\kappa,\alpha) \\
	&= \begin{cases}
	\ (\eps-1)^2+N-1 +\al - 2\al \dfrac{2\eps + N-2}{\eps^2 + \al}&{\rm for}\quad \alpha>0\vspace{0.5em} \\ 
	\ (\varepsilon-1)^2+ N-1&{\rm for}\quad \alpha=0\ .
	\end{cases}
\end{split}
\end{equation}

In this setting, we calculate the right-hand side of \eqref{lbound of Hardy}.
In the case $\eps < 1-N/2$, by differentiating \eqref{F(Hardy)} directly with respect to $\al$, we see
\begin{align*}
	\frac{\pd}{\pd\al}F(\kappa,\al) = 1-2 (2\eps+N-2) \frac{\eps^2+\kappa}{(\eps^2+\kappa+\al)^2} > 0\ .
\end{align*}
Thus $0\le\alpha\mapsto F(\kappa, \al)$ is monotone increasing for each $\kappa>0$, 
and 
\[
	F(\kappa,\al) \ge F(\kappa,0) =(\varepsilon-1)^2+ N-1+\kappa> F(0,0)=F(0,\alpha_0)\ ,
\]
that implies
\[
	\inf_{\kappa>0} \inf_{\nu\in\mathbb{N}\cup\{0\}} F(\kappa,\alpha_\nu) = F(0,\alpha_0)\qquad {\rm when}\quad \eps<1-N/2.
\]
In the case $\eps \ge 1-N/2$, by \eqref{F(Hardy)} we see that $F(\kappa, \al)$ is increasing in $\kappa>0$ for each $\al \ge 0$. 
Thus $ F(\kappa,\al) \ge F(0,\al)$ and
\[ 
	\inf_{\kappa>0}\inf_{\nu\in \mathbb{N}\cup\{0\}}F(\kappa,\alpha_\nu)=\inf_{\nu\in\mathbb{N}\cup\{0\}}F(0,\alpha_\nu).
\]
To evaluate the right-hand side, we compute
\begin{align*}
	\frac{\pd}{\pd\al} F(0,\al) &= 1-2(2\eps + N - 2) \frac{\eps^2}{(\eps^2+\al)^2} = \frac{\eps^4 - 4\eps^3 + 2(\al-(N-2))\eps^2 + \al^2}{(\eps^2 + \al)^2} \\
	&\ge \frac{\varepsilon^2(\varepsilon+2)^2+\al^2}{(\eps^2 + \al)^2} > 0\qquad {\rm if}\ \ \alpha\ge N\ .
\end{align*}
Thus we have $F(0,\al) > F(0,N)$ for any $\al \ge N$, 
which implies
$F(0,\al_\nu) \ge F(0,\al_2) = F(0,2N)$ for all $\nu \ge 2$.
This in turn implies 
\[
	\inf_{\nu\in\mathbb{N}\cup\{0\}}F(0,\alpha_\nu)=\min_{\nu\in\{0,1,2\}}F(0,\alpha_\nu)\ .
\]
Moreover, by computing
\[
\begin{split}
	F(0,\alpha_2)-F(0,\alpha_1)
	&=F(0,2N)-F(0,N-1) \\
	&= \frac{(N+1) \eps^2 \big( (\eps- 2)^2 + N-1 \big) + 2N(N-1)}{(\eps^2 + N-1)(\eps^2 + 2N)} > 0\ ,\ 
\end{split}
\]
we see that
\[
	\inf_{\nu \in \{ 0, 1, 2 \}} F(0,\alpha_\nu) = \min_{\nu \in \{0,1\}} F(0,\alpha_\nu).
\]
Therefore, by calculating 
 \begin{align*}
	&F(0,\al_1) - F(0,\al_0) =F(0,N-1)-F(0,0)= (N-1) \frac{(\eps -2)^2-(N+1)}{\eps^2 + N-1},
\end{align*}
it turns out that 
\begin{equation}
\begin{split}
  \label{minF}	
	\inf_{\kappa > 0}\inf_{\nu \in \N \cup \{0\}} F(\kappa,\al_\nu) &=\min_{\nu\in\{0,1\}}F(0,\alpha_\nu)\\
	&=
	\begin{cases}
	F(0,\al_1) &{\rm for}\quad (\varepsilon-2)^2\le N+1\ , \\ 
	F(0,\al_0)  & {\rm for}\quad (\varepsilon-2)^2>N+1
\end{cases}\end{split}
\end{equation}
when $\varepsilon\ge1-N/2$.
The expression \eqref{minF} holds true even for $\eps < 1-N/2$ since $\varepsilon<1-N/2$ implies $(\varepsilon-2)^2> N+1$.

Inserting this result into \eqref{lbound of Hardy}, we have
\begin{align*}
	H_{N,\gamma} & \ge \min_{\nu\in\{0,1\}}F(0,\alpha_\nu) \\ 
	&=\begin{cases}
	F(0,\al_1) =(\eps -1)^2 \frac{\eps^2+3(N-1)}{\eps^2+N-1} \quad & \text{for} \quad |\eps -2|\le\sqrt{N+1}\ , \vspace{0.5em}\\ 
	F(0,\al_0) = (\eps -1)^2+N-1 & \text{otherwise}.
	\end{cases}
\end{align*} 
Returning to $\eps = 2 - \frac{N}{2}-\gamma$, we arrive at the desired infimum value in Theorem \ref{Theorem:Hardy-Leray}.

%
%
\vspace{1em}\noindent
\subsection*{Optimality for $H_{N,\gamma}$}

In this subsection, we prove that the former lower bound of $H_{N, \gamma}$ is indeed realized as equality: 
\[
	H_{N,\gamma} = \min_{\nu\in\{0,1\}}F(0,\alpha_\nu)=\min_{\nu \in \{0,1\}}\lim_{|\lambda|\searrow+0} \frac{Q_1(\la,\al_\nu)}{P_1(\la,\al_\nu)}.
\]
For that purpose, let $\nu_0 \in \{0,1\}$ be such that
\[
	\min_{\nu \in \{0,1\}} F(0,\al_\nu) = F(0, \al_{\nu_0}).
\]
By the argument in the last subsection, it is enough to prove that
there exists a sequence of vector fields $\{ \bu_n \}_{n \in \N} \subset C_c^{\infty}(\re^N)^N$ such that
\begin{equation}
\label{Hardy_optimal}
	\lim_{n \to \infty} \frac{\int_{\re^N} |\bx|^{2\gamma}|\nabla \bu_n|^2 d{x}}{\int_{\re^N} |\bx|^{2\gamma-2}|\bu_n|^2 d{x}} 
	= \lim_{|\la| \searrow +0} \frac{Q_1(\la,\al_{\nu_0})}{P_1(\la,\al_{\nu_0})}.
\end{equation}

For the construction of $\{ \bu_n \}_{n \in \N}$,
take any nonnegative $h \in C_c^{\infty}(\re)$, $h \not\equiv 0$ and put $h_n(t) = h(t/n)$ for $n \in \N$.
Set
\begin{equation}
\label{test_v}
	\bv_n(\rho,{\bsig}) = \be_\rho \(\eps h_n(t) + h_n'(t)\) \psi_{\nu_0}(\bsig) + h_n(t) \nabla_{\bsig} \psi_{\nu_0}(\bsig)
\end{equation}
where $\rho = e^t$ and $\psi_{\nu_0}$ denotes an eigenfunction of $-\Delta_{\bsig}$ associated with the eigenvalue $\al_{\nu_0} = \nu_0(\nu_0+N-2)$.
Then it is clear that $\bv_n$ satisfies \eqref{curl-free-v(component)}.
Define
\begin{equation}
\label{test_u}
	\bu_n(\rho, \bsig) = \rho^{\eps-1} \bv_n(\rho, \bsig)
\end{equation}
for $\eps = 2 - N/2 - \gamma$.
Then $\{ \bu_n \}_{n \in \N}$ is a sequence of curl-free vector fields having compact supports in $\re^N \setminus \{{\bm 0}\}$.
Put
\[
	 f_n(\la,\bsig) = \wh{(v_n)_\rho}(\la,\bsig) = \(\eps +i\la \)\wh{h_n}(\la) \psi_{\nu_0}(\bsig)
\]
and compute the Hardy-Leray quotient for $\bu_n$ by using \eqref{Hardy-RHS} and \eqref{Hardy-LHS}.
We see
\begin{align*}
	\frac{\int_{\re^N} |\bx|^{2\gamma}|\nabla \bu_n|^2 d{x}}{\int_{\re^N} |\bx|^{2\gamma-2}|\bu_n|^2 d{x}} 
	&=  \frac{\iint_{\Sp^{N-1} \times \re} \ol{f_n} Q_1(\la,-\Delta_{\bsig})f_n d\sigma d\la}
	{\iint_{\Sp^{N-1} \times \re} \ol{f_n} P_1(\la,-\Delta_{\bsig})f_n d\sigma d\la} \notag \\
	&=  \frac{\int_{\re} (\eps^2 + \la^2) Q_1(\la,\al_{\nu_0}) |\wh{h_n}(\la)|^2 d\la}
	{\int_{\re} (\eps^2 + \la^2) P_1(\la,\al_{\nu_0}) |\wh{h_n}(\la)|^2 d\la} \notag \\
	&= \frac{\int_{\re} Q_{01}(\la,\al_{\nu_0}) |\wh{h_n}(\la)|^2  d\la}{\int_{\re} P_{01}(\la,\al_{\nu_0}) |\wh{h_n}(\la)|^2  d\la},
\end{align*}
here
\begin{align}
\label{P_01}
	&P_{01}(\la,\al)=(\eps^2+\la^2) P_1(\la,\al) = \eps^2 + \al + \la^2, \\
	&Q_{01}(\la,\al)=(\eps^2 + \la^2) Q_1(\la,\al) \notag
\end{align}
are polynomials of $\la$.
Note that $\wh{h_n}(\la) = \wh{h(t/n)}(\la) = n\wh{h}(n\la)$.
Thus if $\eps^2 + \al_{\nu_0} \ne 0$, then we have
\begin{align*}
	\frac{\iint_{\re^N}|\nabla {\bm u}_n|^2|x|^{2\gamma}dx}{\iint_{\re^N}|{\bm u}_n|^2|x|^{2\gamma-2}dx}
	&=\frac{\int_{\re} Q_{01}(\la,\al_{\nu_0})|\wh{h}(n\la)|^2d\la}{\int_{\re}P_{01}(\la,\al_{\nu_0})|\wh{h}(n\la)|^2d\la} \\
	&\to \frac{Q_{01}(0,\al_{\nu_0})}{P_{01}(0,\al_{\nu_0})} = \lim_{|\la|\to +0}\frac{Q_1(\la,\al_{\nu_0})}{P_1(\la,\al_{\nu_0})}
\end{align*}
as $n \to \infty$.
In the case $\eps = 0 = \al_{\nu_0}$, by using
\begin{align*}
   P_{01}(\la,0) = \la^2, \quad Q_{01}(\la,0) = N\la^2 + \la^4, 
\end{align*}
we can check that
 \begin{align*}
	\frac{\iint_{\re^N}|\nabla \bu_n|^2 |x|^{2\gamma}dx}{\iint_{\re^N}|\bu_n|^2 |x|^{2\gamma-2}dx}
	&=\frac{\int_{\re} Q_{01}(\la,0)|\wh{h}(n\la)|^2d\la}{\int_{\re} P_{01}(\la,0)|\wh{h}(n\la)|^2d\la}
	= \frac{\int_{\re} (N\la^2+\la^4)|\wh{h}(n\la)|^2d\la}{\int_{\re} \la^2|\wh{h}(n\la)|^2d\la} \\
	&\to N = \lim_{|\la| \to +0} \frac{Q_1(\la,0)}{P_1(\la,0)}
\end{align*}
as $n \to \infty$.
Thus we have proved \eqref{Hardy_optimal} which shows the optimality of $H_{N, \gamma}$ in the class of curl-free vector fields in $C_c^{\infty}(\re^N)^N$.
\qed


%
%
\section{Proof of Theorem \ref{Theorem:Rellich-Leray}}

Let $\gamma \ne 2 - N/2$ be a real number and put $\eps = 3-N/2 - \gamma \ne 1$.
Under the transformation $\bv = \rho^{1- \eps} \bu$ in \eqref{BV}, the gradient vector field is transformed as
\[
	 \nabla \bv = \nabla(\rho^{1-\eps}\bu)=(1-\eps)\rho^{-\eps}\be_\rho \otimes \bu+\rho^{1-\eps} \nabla \bu,
\]
which leads to
\begin{equation}
\label{rho_grad_v}
	 |\rho\nabla \bv|^2 = (1-\eps)^2 | \rho^{1-\eps}\bu |^2 + 2 (1-\eps) \rho^{2-2\eps} \bu \cdot \rho \pd_\rho \bu + \rho^{2-2\eps} |\rho\nabla \bu|^2.
\end{equation}
On the other hand, the assumption $\intRN |\bx|^{2-2\eps-N}|\bu|^2d{x} < \infty$ and the smoothness of $\bu$ imply that
\[
	\bu(\bx) = O \(|\bx|^m \), \quad \nabla \bu(\bx) = O \(|\bx|^{m-1} \) \quad \text{as} \quad |\bx| \searrow  0 
\]
for some integer $m>\eps-1$ if $\eps>1$.
Therefore, we see that $\bv$ must satisfy 
\begin{equation}
\label{BC}
	|\bv(0)| = \lim_{\rho \searrow 0}|\rho \nabla \bv|=0
\end{equation}
by \eqref{rho_grad_v} when $\eps > 1$.

Next, we see the $\Delta \bu$ is written in terms of $\bv$ as follows:
\begin{align}
\label{Delta_u}
	 \Delta \bu = \Delta (\rho^{\eps-1} \bv)
	= \rho^{\eps-3} \(\al_{\eps-1} \bv + (2\eps+N-4) \pd_t \bv+\pd_t^2 \bv + \Delta_{\bsig} \bv \),
\end{align}
here we have used \eqref{Emden-Delta} and $\Delta \rho^{\eps-1} = \al_{\eps-1} \rho^{\eps-3}$.
Note that $\iint_{\Sp^{n-1} \times \re} \pd_t \bv \cdot \bv d\sigma dt = \iint_{\Sp^{n-1} \times \re} \pd_t^2 \bv \cdot \pd_t \bv d\sigma dt = 0$ and
$\iint_{\Sp^{n-1} \times \re} \bv \cdot \pd_t^2 \bv d\sigma dt = -\iint_{\Sp^{n-1} \times \re} |\pd_t \bv|^2 d\sigma dt$ by \eqref{BC}.
Thus by using \eqref{Delta_u}, Lemma \ref{Lemma:integrals},
and noting $(2\eps + N-4)^2 -2 \al_{\eps} = (N-2)^2 + 2\al_{\eps}$,
we find that the both sides of the Rellich-Leray inequality \eqref{Eq:Rellich-Leray} are written as
\begin{align}
\label{Rellich-RHS}
	\intRN  |\bx|^{2\gamma} |\Delta \bu|^2 d{x} &= \intRN  |\bx|^{6-2\eps-N} |\Delta \bu|^2 d{x} \\
	&= \intS  d\sigma \int_0^\infty \big| \al_{\eps-1} \bv 
	+ (2\eps+N-4) \pd_t \bv + \pd_t^2 \bv + \Delta_{\bsig} \bv \big|^2 \frac{d\rho}{\rho} \notag \\
	&= \intS  d\sigma \int_{\re} \Big( \al_{\eps-1}^2 |\bv|^2 + \((N-2)^2+2\al_{\eps-1} \)|\pd_t \bv|^2 + |\pd_t^2 \bv|^2 \notag \\
	&\hspace{30mm}- 2\al_{\eps-1} |\nabla_{\bsig} \bv|^2 + 2 |\pd_t \nabla_{\bsig} \bv|^2 + |\Delta_{\bsig} \bv|^2 \Big) dt \notag \\
	&=\iint_{\re \times \Sp^{N-1}} \bigg( \Big(\al_{\eps-1}^2 + \((N-2)^2 + 2\al_{\eps-1}\) \la^2 + \la^4 \Big)|\bvh|^2 \notag \\
	&\hspace{30mm} + 2\(\la^2-\al_{\eps-1}\) |\nabla_{\bsig} \bvh|^2 + |\Delta_{\bsig} \bvh|^2 \bigg) d\la d\sigma \notag \\
	&=\iint_{\re \times \Sp^{N-1}} \ol{f} \bigg( \(\al_{\eps-1}^2 + \((N-2)^2+2\al_{\eps-1}\) \la^2 + \la^4 \) P_1(\la, -\Delta_{\bsig}) \notag \\
	&\hspace{30mm}+ 2\(\la^2-\al_{\eps-1}\) P_2(\la, -\Delta_{\bsig}) + P_3(\la, -\Delta_{\bsig}) \bigg) f d\la d\sigma, \notag
\end{align}
and
\begin{align}
\label{Rellich-LHS}
	\intRN |\bx|^{2\gamma-4} |\bu|^2 d{x} &= \intRN |\bx|^{2-2\eps-N} |\bu|^2 d{x} \\
	&= \iint_{\re \times \Sp^{N-1}} \ol{f} P_1(\la, -\Delta_{\bsig}) f d\la d\sigma. \notag
\end{align}
Therefore, by \eqref{Rellich-RHS} and \eqref{Rellich-LHS}, 
the optimal constant in \eqref{Eq:Rellich-Leray} can be expressed as
\begin{equation}
\label{R_N(New)}
	R_{N,\gamma} = \inf_{\bu \neq 0, {\rm curl} \, \bu = {\bm 0}} \frac{\intRN  |\bx|^{2\gamma}|\Delta \bu|^2d{x}}{\intRN |\bx|^{2\gamma-4}|\bu|^2d{x}}
	= \inf_{f \neq 0} 
\frac{\iint_{\re \times \Sp^{N-1}}\ol{f} Q_2(\la,-\Delta_{\bsig})f d\la d\sigma}{\iint_{\re \times \Sp^{N-1}}\ol{f}P_1(\la, -\Delta_{\bsig})f d\la d\sigma}
\end{equation}
with the polynomial $Q_2(\la,\al)$ given by
\begin{align}
\label{Q_2}
	Q_2(\la,\al) &= \(\al_{\eps-1}^2 + \((N-2)^2 + 2 \al_{\eps-1}\) \la^2 + \la^4 \) P_1(\la, \al) \\ 
	&\quad + 2\(\la^2-\al_{\eps-1}\) P_2(\la,\al) + P_3(\la,\al). \notag
\end{align}

%
%
\vspace{1em}\noindent
\subsection*{Calculation of a lower bound}

As in \eqref{expansion},
we expand $f$ in terms of eigenfunctions of $-\Delta_{\bsig}$.
Then by \eqref{R_N(New)}, \eqref{Q_2}, and Lemma \ref{Lemma:appendix}, we find
\[
	 R_{N,\gamma} \ge \inf_{\nu \in \N \cup \{0\}} \inf_{\la \in \re \setminus \{0\}} \frac{Q_2(\la,\al_\nu)}{P_1(\la,\al_\nu)}
%
	= \inf_{\nu \in \N \cup \{0\}} \inf_{\kappa > 0} F(\kappa,\al_\nu),
\]
where for $\kappa >0$ and $\al \ge 0$, $F(\kappa, \al)$ is defined as
\[
\begin{split}
	F(\kappa,\al) &= \dfrac{Q_2(\sqrt{\kappa},\al)}{P_1(\sqrt{\kappa},\al)}
	\\&=\al_{\eps-1}^2 + \((N-2)^2+2 \al_{\eps-1}\) \kappa + \kappa^2 + \frac{2\(\kappa-\al_{\eps-1}\) P_2(\sqrt{\kappa},\al) 
	+ P_3(\sqrt{\kappa},\al)}{P_1(\sqrt{\kappa},\al)} \ .
\end{split}
\]  
By directly calculating further, we can check that
\begin{align*}
	F(\kappa,\al) 	
	&= \kappa^2 + \frac{4 \al (1-\eps)(N+2 \eps-2)^2 \kappa}{(\eps^2+\al)(\kappa + \eps^2 + \al)} \notag \\
	&\hspace{10mm} + \(\frac{N^2}{2}+2 \(\eps+\frac{N-4}{2} \)^2 + 2 \al \) \kappa + \frac{(\eps-2)^2+\al}{\eps^2+\al} (\al_\eps-\al)^2
\end{align*}
for $\eps = 3 - N/2 - \gamma \ne 0$, and
\[
	 F(\kappa,\al) = \kappa^2 + \frac{4(N-2)^2 \kappa}{\kappa + \al} + \((N-2)^2 + 4 + 2\al\) \kappa + (4 + \al) \al
\]
for $\eps = 0$. 
We also define $F(0,\alpha)$ as
\begin{equation}
\label{F0al}
\begin{split}
	F(0,\alpha)&=\lim_{|\lambda|\searrow+0}\frac{Q_2(\lambda,\alpha)}{P_1(\lambda,\alpha)}=\lim_{\kappa\searrow+0}F(\kappa,\alpha) \\
	&= \begin{cases}
	\frac{(\eps-2)^2+\al}{\eps^2+\al} (\al_\eps-\al)^2, &\qquad \text{for} \, \eps \ne 0, \al \ge 0, \\
	(4 + \al) \al, &\qquad \text{for} \, \eps = 0, \al > 0, \\
	4(N-2)^2, &\qquad \text{for} \, \eps = 0, \al = 0.
	\end{cases}
\end{split}
\end{equation}
In these settings, from now on we evaluate the expression 
\[
	\inf_{\nu \in \mathbb{N} \cup \{0\}} \inf_{\kappa > 0} F(\kappa,\al_\nu). 
\]
If $\eps < 1$, it is clear that the map $0 < \kappa \mapsto F(\kappa,\al)$ is increasing for any fixed $\al \ge 0$. 
Also, if $\eps>1$, estimating $\pd_\kappa F(\kappa,\al)$ from below as
\begin{align*}
	\frac{\pd F(\kappa,\al)}{\pd\kappa}
	&= 2\kappa - \frac{4\al(\eps-1)(N+2 \eps-2)^2}{(\kappa+\eps^2+\al)^2} + \frac{N^2}{2}+2\(\eps+\frac{N-4}{2}\)^2 + 2\al \\
	&\ge -\frac{4\al(\eps-1)(N+2 \eps-2)^2}{(\eps^2+\al)^2} + \frac{N^2}{2} + 2\(\eps+\frac{N-4}{2}\)^2 + 2\al \\
	&\ge -\frac{\eps-1}{\eps^2} (N+2 \eps-2)^2 + \frac{N^2}{2} + 2\(\eps+\frac{N-4}{2}\)^2+2 \al \\
	&\ge -\frac{1}{4} (N+2 \eps-2)^2 + \frac{N^2}{2} + 2\(\eps+\frac{N-4}{2}\)^2+2 \al \\
	&= \(\eps+\frac{N}{2}-3\)^2 + \frac{N^2-4}{2}+2 \al \ge 0,
\end{align*}
we see again $F(\kappa,\al)$ is increasing with respect to $\kappa > 0$ for any $\al \ge 0$. 
Therefore we have
\[
	\inf_{\kappa > 0} F(\kappa,\al) = F(0,\al)
\]
for all $\varepsilon\ne1$, which implies
\[
	\inf_{\nu\in\mathbb{N}\cup\{0\}}\inf_{\kappa>0}F(\kappa,\alpha_\nu)=\inf_{\nu\in\mathbb{N}\cup\{0\}}F(0,\alpha_\nu)\ .
\]
Moreover,
we can check that
\[	
	\frac{\pd F(0,\al)}{\pd\al} \ge 0, \quad \al \ge \max \{ \al_1, \al_{\eps} \},
\]
see Lemma \ref{Lemma:A2}.
This implies $\disp{\inf_{\nu\in\mathbb{N}\cup\{0\}}} F(0,\al_{\nu})$ is attained. 
Therefore, we have the desired estimate:
\begin{equation}
\label{Rlower}
	 R_{N,\gamma} \ge \min_{\nu \in \N \cup \{0\}} F(0,\al_\nu) \qquad {\rm with}\quad
	 F(0,\al_\nu) = \lim_{|\lambda|\searrow+0}\frac{Q_2(\lambda,\alpha_\nu)}{P_1(\lambda,\alpha_\nu)}\ .
\end{equation}
Furthermore, we see that $\disp{\min_{\nu \in \N \cup \{0\}}} F(0,\al_\nu)$ is given by 
\[
	\min_{\nu\in\mathbb{N}\cup\{0\}}F(0,\alpha_\nu)= \min_{\nu \in \N \cup \{0\}} \frac{(\eps-2)^2+\al_\nu}{\eps^2+\al_\nu} (\al_\eps-\al_\nu)^2\
\]
for $\eps = 3 -N/2 - \gamma \ne 0$,
and
\begin{align*}
	\min_{\nu\in\mathbb{N}\cup\{0\}}F(0,\alpha_\nu) 
	&= \min \Big\{ 4(N-2)^2, (4 + \al_1)\al_1 \Big\} \\
	&=\begin{cases}
	\ 4(N-2)^2 = F(0,\al_0) &\text{for} \quad N=2,3,4, \\ 
	\ (N+3)(N-1) = F(0,\al_1) &\text{for} \quad N \ge 5
	\end{cases}
\end{align*}
for $\eps = 3 -N/2 - \gamma =0$ . 
This gives the lower bound of $R_{N, \gamma}$.
In the next subsection we will show that the above inequality is indeed the equality.

%
%
\vspace{1em}\noindent
\subsection*{Optimality for $R_{N,\gamma}$}

To show that the inequality \eqref{Rlower} is indeed the equality, 
let $\nu_0\in\mathbb{N}\cup\{0\}$ be such that $F(0,\alpha_{\nu_0})=\disp\min_{\nu\in\mathbb{N}\cup\{0\}}F(0,\alpha_\nu)$ is satisfied. 
We use the sequence of curl-free vector fields $\{ \bu_n \}_{n \in \N}$ in \eqref{test_u} again with \eqref{test_v},
however for $\eps = 3 - N/2 - \gamma$.
Then, as in the proof of Theorem \ref{Theorem:Hardy-Leray}, we obtain the following expression:
\begin{align*}
	\frac{\int_{\re^N} |\bx|^{2\gamma}|\Delta \bu_n|^2 d{x}}{\int_{\re^N} |\bx|^{2\gamma-4}|\bu_n|^2 d{x}} 
	&=  \frac{\iint_{\Sp^{N-1} \times \re} \ol{f_n} Q_2(\la,-\Delta_{\bsig})f_n d\sigma d\la}
	{\iint_{\Sp^{N-1} \times \re} \ol{f_n} P_1(\la,-\Delta_{\bsig})f_n d\sigma d\la} \notag \\
	&=  \frac{\int_{\re} (\eps^2 + \la^2) Q_2(\la,\al_{\nu_0}) |\wh{h_n}(\la)|^2 d\la}
	{\int_{\re} (\eps^2 + \la^2) P_1(\la,\al_{\nu_0}) |\wh{h_n}(\la)|^2 d\la} \notag \\
	&= \frac{\int_{\re} Q_{02}(\la,\al_{\nu_0}) |\wh{h_n}(\la)|^2  d\la}{\int_{\re} P_{01}(\la,\al_{\nu_0}) |\wh{h_n}(\la)|^2  d\la},
\end{align*}
where $P_{01}(\la,\al)$ is as in \eqref{P_01} and
\begin{align*}
	Q_{02}(\la,\al) = (\eps^2+\la^2) Q_2(\la,\al)
\end{align*}
is a polynomial of $\la$.
When $\eps = 0$ and $\al_{\nu_0} = 0$,
by using the facts
\[
	Q_{02}(\la,0) =4(N-2)^2\la^2 + (N^2-4N+8)\la^4 + \la^6
\]
and $P_{01}(\la,0) = \la^2$, we prove that
\[
	\lim_{n \to \infty} \frac{\int_{\re^N} |\bx|^{2\gamma}|\Delta \bu_n|^2 d{x}}{\int_{\re^N} |\bx|^{2\gamma-4}|\bu_n|^2 d{x}} 
	= 4(N-2)^2 = F(0,0).
\]
Thus as in the proof of Theorem \ref{Theorem:Hardy-Leray}, we can show that
\[
	\lim_{n \to \infty} \frac{\int_{\re^N} |\bx|^{2\gamma}|\Delta \bu_n|^2 d{x}}{\int_{\re^N} |\bx|^{2\gamma-4}|\bu_n|^2 d{x}} 
	=F(0,\alpha_{\nu_0})
\]
for all cases $\eps^2 + \al_{\nu_0} \ne 0$ and $\eps^2 + \al_{\nu_0} = 0$.
This leads to the optimality of $R_{N, \gamma}$.
\qed


\section{Appendix.}

In this appendix, we prove technical lemmas.

\begin{lemma}
\label{Lemma:appendix}

Let $(X, \mathcal{M}, \mu)$ be a measure space and let $\xi, \eta: X \to \re$ be a $\mu$-measurable function such that $\xi \ne 0$ $\mu$-a.e.
Suppose $g: X \to \re$ is a $\mu$-measurable function satisfying, $\xi g \ge 0$ $\mu$-a.e., $0 < \int_X \xi g d\mu < \infty$, and $\int_X |\eta g| d\mu < \infty$.
Then we have
\[
	\frac{\int_X \eta g d\mu}{\int_X \xi g d\mu} \ge {\rm ess} \, \inf_{x \in X} \frac{\eta(x)}{\xi(x)}.
\]
\end{lemma}

\begin{proof}
Let $I = {\rm ess} \, \inf_{x \in X} \frac{\eta(x)}{\xi(x)}$.
Then $\frac{\eta}{\xi} \ge I$ $\mu$-a.e. 
Multiply the both sides by $\xi g \ge 0$, we have $\eta g = \frac{\eta}{\xi} \xi g \ge I \xi g$ $\mu$-a.e..
By integrating over $X$, we obtain
\[
	\int_X \eta g d\mu \ge I \int_X \xi g d\mu
\]
which leads the result.
\end{proof}

\begin{lemma}
\label{Lemma:A2}
Let $F(0, \al)$ be defined by \eqref{F0al}.
Then we have
\[	
	\frac{\pd F(0,\al)}{\pd\al} \ge 0 \quad \text{for} \quad \al \ge \max \{ \al_1, \al_{\eps} \}. 
\]
\end{lemma}

\begin{proof}
Recall $\al_1 = N-1$ and $\al_\eps = \eps(\eps+N-2)$. 
It is enough to show Lemma when $\eps \ne 0$ and $F(0, \al) = \frac{(\eps-2)^2+\al}{\eps^2+\al} (\al_\eps-\al)^2$.
A direct computation shows that
\begin{align*}
	&\frac{\pd F(0, \al)}{\pd \al} = \frac{2(\al - \al_{\eps})}{(\al + \eps^2)^2} f_{\eps}(\al), \quad \text{where} \\
	&f_{\eps}(\al) = \al^2 + 2(\eps^2 - \eps + 1) \al + \eps^2(\eps-1)^2 + 2\al_{\eps}(1-\eps). 
\end{align*}
Since $\eps^2 - \eps + 1 >0$ for any $\eps \in \re$, we see that $f_{\eps}$ is strictly increasing for $\al \ge 0$.
Thus if we show 
(i) $f_\eps(\al_\eps) \ge 0$ if $\al_\eps \ge \al_1$, and (ii) $f_{\eps}(\al_1) \ge 0$ if $\al_1 \ge \al_\eps$,
then $f_{\eps}(\al) \ge 0$ for any $\al \ge \max \{ \al_1, \al_\eps \}$, which concludes Lemma.

To prove (i), we observe that $f_{\eps}(\al_{\eps}) = (\al_{\eps} + \eps^2)(\al_{\eps} + (\eps-2)^2)$.
Thus if $\al_{\eps} \ge \al_1 = N-1 > 0$, clearly we have $f_{\eps}(\al_{\eps}) > 0$.

To prove (ii), we observe that $f_{\eps}(\al_1) = f_{\eps}(N-1) = \eps^4 - 6\eps^3 + 8\eps^2 -2 \eps + N^2-1$.
We need to prove this quartic function is nonnegative for $\eps \in \re$ such that $\al_1 \ge \al_\eps$, i.e., $-(N-1) \le \eps \le 1$.
However, this is an elementary fact.
\end{proof}

%
%

\vspace{1em}\noindent
{\bf Acknowledgments.}

The second author (F.T.) was supported by JSPS Grant-in-Aid for Scientific Research (B), No.15H03631.

\end{document}